\documentclass[11pt]{article}
\usepackage[margin=1in]{geometry}
\usepackage{amssymb,amsfonts,amsmath,bbm,mathrsfs,stmaryrd,mathtools}
\usepackage{xcolor}
\usepackage{url}

\usepackage{enumerate}

\usepackage[backref=page]{hyperref}
\hypersetup{colorlinks,
             linkcolor=black!75!red,
             citecolor=blue,
             pdftitle={},
             pdfproducer={pdfLaTeX},
             pdfpagemode=None,
             bookmarksopen=true
             bookmarksnumbered=true}

\usepackage{tikz}
\usetikzlibrary{arrows,calc,decorations.pathreplacing,decorations.markings,intersections,shapes.geometric,through,fit,shapes.symbols,positioning,decorations.pathmorphing}

\usepackage{braket}
\usepackage{tensor}

\usepackage[amsmath,thmmarks,hyperref]{ntheorem}
\theoremstyle{plain}

\newtheorem{lemma}{Lemma}[section]
\newtheorem{proposition}[lemma]{Proposition}
\newtheorem{corollary}[lemma]{Corollary}
\newtheorem{theorem}[lemma]{Theorem}

\theoremstyle{plain}
\theoremnumbering{Alph}

\theoremstyle{plain}
\theorembodyfont{\upshape}
\theoremsymbol{\ensuremath{\blacklozenge}}

\newtheorem{definition}[lemma]{Definition}
\newtheorem{example}[lemma]{Example}
\newtheorem{examples}[lemma]{Examples}
\newtheorem{remark}[lemma]{Remark}
\newtheorem{remarks}[lemma]{Remarks}

\numberwithin{equation}{section}

\theoremstyle{nonumberplain}
\theoremsymbol{\ensuremath{\blacksquare}}

\newtheorem{proof}{Proof}

\newcommand\bZ{{\mathbb Z}}

\newcommand\cT{{\mathcal T}}

\DeclareMathOperator{\id}{id}

\newcommand{\qedhere}{\mbox{}\hfill\ensuremath{\blacksquare}}

\definecolor{ForestGreen}{RGB}{20,130,40}

\title{A categorical approach to the set-theoretic Yang-Baxter equation}

\author{A.L. Agore, A. Chirvasitu and G. Militaru}

\begin{document}

\date{}

\newcommand{\Addresses}{{% additional braces for segregating \footnotesize
  \bigskip
  \footnotesize

  %\medskip

\author{A. L. Agore:}
 \textsc{Simion Stoilow Institute of Mathematics of the Romanian Academy}
  \par\nopagebreak
  \textsc{Bucharest, P.O. Box 1-764, 014700, Romania}
  \par\nopagebreak
  \textit{E-mail address}: \texttt{ana.agore@gmail.com}

  \medskip

  \author{A. Chirvasitu:}
  \textsc{Department of Mathematics, University at Buffalo}
  \par\nopagebreak
  \textsc{Buffalo, NY 14260-2900, USA}
  \par\nopagebreak
  \textit{E-mail address}: \texttt{achirvas@buffalo.edu}

  \medskip

 \author{G. Militaru:}
  \textsc{Faculty of Mathematics and Computer Science, University of Bucharest}
  \par\nopagebreak
  \textsc{Str. Academiei 14, RO-010014 Bucharest 1, Romania}
  \par\nopagebreak
  \textit{E-mail address}: \texttt{gigel.militaru@gmail.com}

  % % \medskip
  % %
  % % \textsc{Department of Mathematics, INSTITUTION}
  % % \par\nopagebreak
  % % \textsc{ADDRESS}
  % % \par\nopagebreak
  % % \textit{E-mail address}: \texttt{??}
  % %

}}

\maketitle

\bigskip

\begin{abstract}
We investigate set-theoretic solutions of the Yang--Baxter equation from a categorical perspective. The paper has two main objectives. On the one hand, several existing constructions and classification results in the literature are recast in categorical terms, thereby showing that certain hypotheses previously imposed on such solutions are in fact superfluous. On the other hand, the categorical approach leads naturally to new classes of Yang--Baxter solutions. Our approach consists in identifying suitable subcategories of the category $\mathcal{YB}$ of Yang--Baxter solutions and providing categorical characterisations of the solutions they encode. As an application, we prove that the category of P{\l}onka bi-magmas is isomorphic to the subcategory of $\mathcal{YB}$ consisting of solutions of the Baaj--Long--Skandalis equation. Along the way, we introduce and study a new class of Lyubashenko type solutions and examine the extent to which non-degenerate bijective solutions of the braid equation arise from skew left braces.
\end{abstract}


\vspace{.5cm}


%\tableofcontents

%%%%%%%%%%%%%%%%%%%%%%%%%%%%%%%%%%%%%%%%%%%%%%%%%%%%%%%%%%%%%%%%%%%%%%%%%%%%%
%%%%%%%%%%%%%%%%%%%%%%%%%%%%%%%%%%%%%%%%%%%%%%%%%%%%%%%%%%%%%%%%%%%%%%%%%%%%%
\section*{Introduction}
The origins of the quantum Yang-Baxter equation (QYBE) can be traced back to the independent contributions of Yang \cite{Yang}, who studied the one-dimensional many-body problem, and Baxter \cite{baxter}, whose work focused on statistical mechanics. Despite its fundamental role in mathematical physics and the theory of quantum groups, the problem of finding and classifying all solutions of the QYBE remains largely open. In this context, the idea of studying a set-theoretic version of the Yang-Baxter equation was first suggested by Drinfel'd in \cite{dr}. More precisely, given a set $X$, Drinfel'd asked for the classification of all maps $R: X \times X \to X \times X$ satisfying the equation:
\begin{equation}\label{qybe2}
R^{12} R^{13} R^{23} = R^{23}R^{13}R^{12}
\end{equation}
where $R^{ij} : X \times X \times X \to X \times X \times X$ denotes $R$ acting in the $i^{th}$ and $j^{th}$ components. In recent years, both the above set-theoretic Yang-Baxter equation and its equivalent braid equation
\begin{equation}\label{braid}
R^{12}R^{23}R^{12} = R^{23}R^{12}R^{23}
\end{equation}
have received significant attention, owing both to the rich algebraic structures underlying their solutions and to their applications across mathematics and physics.\\
Since a complete classification of all Yang-Baxter solutions on a given set $X$ is generally out of reach, the standard strategy in the literature has been to construct broad families of solutions by endowing $X$ with additional algebraic structures (such as a (quasi-)group structure) and imposing identities chosen so that a certain map $R$, induced by those algebraic structures,
 satisfies \eqref{qybe2} or \eqref{braid}. Prototypical examples include cycle sets \cite{rump}, (bi)racks \cite{PPZ}, skew left braces \cite{vendramin1}, trusses \cite{brz}, post-groups \cite{bai_guo}, and the more recent P{\l}onka bi-magmas \cite{ACMv3}. Each of these constructions yields a class of solutions of the Yang-Baxter (resp. braid) equation by means of an auxiliary structure.\\
The present paper develops a categorical approach to the set-theoretic Yang-Baxter equation. A number of results of a categorical nature have already appeared in the literature. For example, the foundational work of Etingof, Schedler and Soloviev \cite[Theorem 2.10]{etingof} establishes an equivalence between the category of non-degenerate, bijective, and involutive solutions of the braid equation (referred to therein as \emph{non-degenerate symmetric sets}) and the category of bijective cocycle quadruples. In the same vein, the more recent paper \cite[Theorem 2.9]{acm1} proves an equivalence between the category of Frobenius-Separability type solutions of the braid equation and the category of pointed Kimura semigroups. Several other results from the Yang-Baxter literature admit a categorical interpretation, although this viewpoint is not generally made explicit. In particular, \cite[Theorem 9]{LuYZ}, which associates a universal group with a braiding operator to each non-degenerate solution of the braid equation, can be reformulated, via the equivalent language of skew left braces \cite[Theorem 4.5]{smo_ven}, as a left adjoint to an inclusion functor. Likewise, Rump's correspondence \cite[Proposition 1]{rump} between cycle sets and left non-degenerate unitary solutions of the Yang-Baxter equation is in fact an isomorphism of categories. We return to these examples in Section~\ref{se:freyd}.\\
The purpose of this paper is twofold. First, we survey and reinterpret several results from the literature in categorical terms; this new perspective allows us to remove certain unnecessary restrictions that were previously imposed on these solutions (see e.g.\ Corollary~\ref{cor:primaadjreform}, where the bijectivity and non-degeneracy hypotheses imposed in \cite{LuYZ} are shown to be superfluous for the existence of the left adjoint). Second, using categorical methods, we establish the existence of new classes of solutions and provide a complete structural description of one such class, including a precise characterisation of the Yang-Baxter solutions it gives rise to. Concretely, our approach rests on the following two-step strategy, which we expect can be applied to many other constructions of the type surveyed here. For stating the problems we use the language of bi-magmas recalled in Section~\ref{se:prel}. \vspace{0.5cm}\\
{\bf Problem A}:  \emph{Construct categories $\mathcal{C}$ (ideally varieties of algebras, or subcategories thereof, in the sense of universal algebra) together with a functor
\[
F \colon \mathcal{C} \longrightarrow {}_{\rm YB}{\rm BiMag} \, \bigl( \cong \mathcal{YB} \bigr)
\]
to the variety of Yang-Baxter bi-magmas (whose isomorphism with the category $\mathcal{YB}$ of solutions of the Yang-Baxter equation is recalled in \eqref{isoybbimag} below).}  \vspace{0.5cm}\\
Such (sub)categories give rise to new families of solutions of the
Yang-Baxter equation and, in certain cases, provide a sufficiently flexible framework for deriving structural and classification results for important classes of objects.
Moreover, $F$ being a genuine \emph{functor}, rather than a mere assignment of a Yang-Baxter solution $(X,R)$ to each object of $\mathcal{C}$, guarantees that morphisms of $\mathcal{C}$ (maps preserving the extra operations) induce morphisms of the associated solutions of the Yang-Baxter equation. Therefore, constructions such as free objects, quotients, products, and other (co)limits in $\mathcal{C}$ are automatically transported to the corresponding solutions along $F$. When $\mathcal{C}$ is a variety of algebras, Freyd's adjoint functor theorem (Corollary~\ref{cor:Freyd}) turns this into a systematic source of universal constructions (e.g. free solutions, universal enveloping objects) without requiring a case-by-case argument. \vspace{0.5cm}\\
{\bf Problem B}:  \emph{Identify which solutions of the Yang-Baxter equation are produced by $F$.}\vspace{0.5cm}\\ More precisely, we aim to characterize the essential image of $F$ intrinsically, in terms of properties of $R$ itself, and decide whether $F$ is a full embedding, or even an isomorphism, onto the corresponding subcategory of $\mathcal{YB}$ consisting of all solutions of the Yang-Baxter equation. In particular, this allows us to determine precisely how much of the category $\mathcal{YB}$ each construction reaches. \\
We successfully apply this strategy to the recently introduced P{\l}onka bi-magmas \cite[Definition 3.7]{ACMv3}.\\
\textbf{Outline of the paper.} Section~\ref{se:prel} sets up the categorical framework used throughout. In particular, it recalls the isomorphism between the categories of Yang-Baxter solutions and Yang-Baxter bi-magmas and places several known classes of solutions (e.g. right/left P{\l}onka magmas, P{\l}onka bi-magmas) as subcategories of Yang-Baxter (resp. braid) solutions.\\
Section~\ref{se:bls} carries out the strategy outlined above for P{\l}onka bi-magmas (\cite[Definition 3.7]{ACMv3}). We begin by introducing, in Definition~\ref{def:longdef}, the subcategory $\mathcal{BLS} \subseteq \mathcal{YB}$ consisting of solutions of the \emph{Baaj-Long-Skandalis} (BLS, for short) equation. The main result of the section is Theorem~\ref{teorema2} which establishes an isomorphism between the category of P{\l}onka bi-magmas and the category $\mathcal{BLS}$ of solutions of the BLS equation. Theorem~\ref{th:ultima} combines Theorem~\ref{teorema2} with the structure theorem for P{\l}onka bi-magmas to give an explicit combinatorial description of all BLS-solutions in terms of set partitions and commuting families of self-maps. Furthermore, the new class of Lyubashenko type solutions is introduced (Definition~\ref{def:lymag}). It is shown to strictly contain both the diagonal solutions and the BLS-solutions, while being strictly contained in $\mathcal{YB}$ (Example~\ref{ex:notlyb}, Corollary~\ref{cor:nolimcolim}). \\
Section~\ref{se:freyd} develops the categorical framework around the Yang-Baxter (resp. braid) equation. First we revisit the intensively studied class of (left-right) non-degenerate solutions of the braid equation and address Problem B for the functor $\mathcal{S}$ sending a skew left brace to its corresponding (left-right) non-degenerate bijective solution of the braid equation. Proposition~\ref{prop_im} shows that its essential image
is a proper subcategory of $\mathcal{BR}^{b,n}$, the category of all bijective and (left-right) non-degenerate solutions of the braid equation, by constructing a large class of bijective and (left-right) non-degenerate solutions of the braid equation which are never isomorphic to solutions arising from skew left braces. We then turn to the general categorical machinery around Freyd's adjoint functor theorem: after recasting Rump's correspondence between cycle sets and left non-degenerate unitary solutions as an isomorphism of categories (Proposition~\ref{cyc_iso}), we establish the existence of free YB-solutions and free bijective YB-solutions (Corollaries~\ref{cor:freesolgen}, ~\ref{co:BLSadj} and Remarks~\ref{constrfreemag}(\ref{bijeuniv})). Moreover, we reformulate the Lu-Yan-Zhu \cite{LuYZ} construction of the universal skew left brace associated to a braid solution as a left adjoint to the functor ${\rm SLB} \to \mathcal{BR}$ (Corollary~\ref{cor:primaadjreform}), removing the bijectivity and non-degeneracy hypotheses used in the original proof. The relevant existence results for various classes of Yang-Baxter solutions of interest are all consequences of Freyd's adjoint functor theorem (see Corollary~\ref{cor:freesolgen} and Corollary~\ref{cor:primaadjreform}).\\
Section~\ref{se:conclusion} outlines several open problems and possible directions for future research. Motivated by the broad range of one-sided and two-sided quasi-group conditions that can be imposed on the two magmas underlying a solution of the Yang-Baxter (resp. braid) equation, Problem~C points out the non-degeneracy conditions not addressed in the literature. We also return to the functor sending a skew left brace to its corresponding (left-right) non-degenerate bijective solution of the braid equation of Section~\ref{se:freyd}, and record as an open problem the characterization of its essential image.

\subsection*{Acknowledgements}
Some of the results in this paper were first announced in 2023 in the preprint \cite{CMv1}, which is not intended for publication and will remain available only on arXiv. In preparing the final version, we used ChatGPT (OpenAI) and Claude (Anthropic, Claude Opus~5.5) for proofreading, correcting typos, and making minor improvements. Claude Opus 5.5 identified a mistake in Example 2.13 and Proposition 2.14 of \cite{CMv1}. The mistake stems from a misprint in \cite[Theorem 2.5 (iii)]{etingof2}, where the stated solution satisfies the Yang-Baxter equation, whereas the solutions presented in parts (i) and (ii) of the same theorem satisfy the braid equation. This misprint does not affect the proof of \cite[Theorem 2.5]{etingof2}. The aforementioned example and proposition have been corrected and are included here as Example~\ref{ex:notlyb} and Proposition~\ref{pr:nomor}, respectively.
%%%%%%%%%%%%%%%%%%%%%%%%%%%%%%%%%%%%%%%%%%%%%%%%%%%%%%%%%%%%%%%%%%%%%%%%%%%%%
%%%%%%%%%%%%%%%%%%%%%%%%%%%%%%%%%%%%%%%%%%%%%%%%%%%%%%%%%%%%%%%%%%%%%%%%%%%%%
\section{Preliminaries}\label{se:prel}

Given two sets $X$ and $Y$, we denote by $\tau_{X, Y} \colon X\times Y \to Y \times X$ the flip map defined by $\tau_{X, Y} \, (x, y) = (y, x)$, for all $x\in X$ and $y\in Y$. If $R\colon X \times X \to X\times X$ is a map, we define:
\begin{eqnarray*}
  && R^{12} ,\,\, R^{23},\,\, R^{13} : X \times X \times X \to X \times X \times X\\
  && R^{12} := R \times {\rm Id}_X, \quad  R^{23} := {\rm Id}_X \times R, \quad R^{13}: = ({\rm Id}_X \times \tau_X) \circ R^{12} \circ ({\rm Id}_X \times \tau_X)
\end{eqnarray*}
where $\tau_X = \tau_{X, X}$. Furthermore, we write $\pi_1$, $\pi_2 \colon X\times X \to X$ for the canonical projections defined by  $\pi_1 (x, y) = x$, $\pi_2 (x, y) = y$ for all $x$, $y \in X$. A map
$R \colon X \times X \to X \times X$ is called \emph{unitary} if $R^{21} R = {\rm Id}_{X \times X}$, where $R^{21} := \tau_X R \tau_X$. Furthermore, $R$ is called of \emph{diagonal type} if $R(x, x) = (x, x)$,
for all $x\in X$. \\
Recall that a map $R: X\times X \to X\times X$ is called a
set-theoretic solution of the \emph{Yang-Baxter equation} if (\ref{qybe2}) holds. In the sequel, a map $R$ satisfying equation (\ref{qybe2}) will be called
briefly a \emph{YB-solution on $X$}. Similarly, the map $R: X\times X \to X\times X$  is called a
set-theoretic solution of the \emph{braid equation} if (\ref{braid}) holds. Both equations feature prominently in the literature devoted to the set-theoretic Yang-Baxter equation, as they are in fact equivalent. Indeed, it can be easily seen that $R$ is a solution of the braid equation (\ref{braid}) if and only if $\tau_{X} \circ R$ is a YB-solution. \\
Throughout we denote by $\mathcal{YB}$ the category of solutions of (\ref{qybe2}), i.e. objects of $\mathcal{YB}$ are pairs $(X, \, R)$, where $X$ is a set and $R$ a YB-solution on $X$. A morphism $\sigma: (X, \, R) \to (X', \, R')$ in $\mathcal{YB}$ is a map $\sigma: X \to X'$ such that $(\sigma \times \sigma) \circ R = R' \circ (\sigma \times \sigma)$. Note that this is equivalent to $\sigma: X \to X'$ being a morphism between the bi-magmas associated: i.e. $\sigma (x \cdot y) = \sigma(x) \cdot' \sigma(y)$ and $\sigma (x \ast y) = \sigma(x) \ast' \sigma(y)$, for all $x$, $y\in X$ (see (\ref{notatieR}) below). Two solutions $(X, \, R)$ and $(X', \, R') \in \mathcal{YB}$ are called \emph{isomorphic} if they are isomorphic as objects in the category $\mathcal{YB}$. For a given YB-solution $(X, \, R) \in \mathcal{YB}$ we denote by ${\rm Aut} \, (X, \, R)$ its automorphism group in the category $\mathcal{YB}$.\\
In the same fashion, let $\mathcal{BR}$ denote the category of solutions of (\ref{braid}), i.e. objects of $\mathcal{BR}$ are pairs $(X, \, R)$, where $X$ is a set and $R$ is a solution on $X$ of the braid equation. A morphism $\sigma: (X, \, R) \to (X', \, R')$ in $\mathcal{BR}$ is a map $\sigma: X \to X'$ such that $(\sigma \times \sigma) \circ R = R' \circ (\sigma \times \sigma)$.\\
The functor
\begin{equation}\label{iso_br_yb}
\mathcal{YB} \to \mathcal{BR}, \qquad (X, \, R) \mapsto (X, \, \tau_X \circ R)
\end{equation}
is obviously an isomorphism of categories with inverse given by $(X, \, R) \mapsto (X, \, \tau_X \circ R)$.\\
\textbf{Varieties of algebras and Freyd's theorem.} For the basic notions of universal algebras we refer the reader to
\cite{adamek_ros, bergman, bruck, buris, LauD} and for category theory to \cite{Ago, mcl}. We recall here only the main concepts that will be used throughout the paper. A  \emph{universal algebra} (or an \emph{algebra}, for short) is a pair $A = (A, \, (\mu_i)_{i\in I})$,
where $A$ is a set and $\mu_i : A^{n_i} \to A$ is an $n_i$-ary operation on $A$ for each $i \in I$, with the integer $n_i \geq 0$  denoting the \emph{arity} of $\mu_i$. Two algebras $A$ and $B$ are said to have the \emph{same signature} (or \emph{type}) if their operations are indexed by the same set $I$ and corresponding operations have the same arity. \\
A \emph{variety of algebras} is a subcategory $\mathcal{V}$ of the category of all universal algebras of a fixed signature whose objects satisfy a prescribed family of identities. Classical examples include the variety of all lattices, (semi)groups, rings, and $R$-modules over a ring $R$. A fundamental theorem of Birkhoff \cite[Theorem 3.9]{adamek_ros} states that a class of universal algebras of the same signature is a variety of
algebras if and only if it is closed under taking subalgebras, direct products, and homomorphic images. Every variety of algebras $\mathcal{V}$ is a (small)
complete and cocomplete category: that is, it admits all (small) limits and colimits \cite[p. 18]{bergman}. For further properties of varieties of algebras we refer to \cite[Remark 3.6 and Corollary 3.7]{adamek_ros}. In particular, every variety of algebras is a \emph{locally finitely presentable
category} \cite[Definition 1.9]{adamek_ros}. A powerful tool in the theory of varieties of algebras is Freyd's famous theorem (\cite{freyd} or \cite[Theorem 8.14, p. 26]{bergman}). The following corollary
 \cite[Corollary 8.17, p. 28]{bergman}, recalled here for convenience, will be used intensively in Section~\ref{se:freyd}.%; to quote George Bergman, it is ``\emph{more widely known that the Theorem itself}''.
\begin{corollary}\label{cor:Freyd}
 Any functor between varieties of algebras which respects underlying sets has a left adjoint. In particular, any forgetful functor $\mathcal{V} \to {\rm Set}$ has a left adjoint.  \qedhere
\end{corollary}
\textbf{Magmas, Yang-Baxter bi-magmas and P{\l}onka bi-magmas.}
The two classes of universal algebras relevant to our study are \emph{magmas} and \emph{bi-magmas}, which we recall below. A \emph{magma} (\cite[Definition 1.1, Chapter IV]{serre} or \cite[Definition 1]{bourbaki}) is a pair $(X, \, \cdot)$ consisting of a set $X$ with a binary operation $\cdot : X \times X \to X$, $(x, \, y) \mapsto x\cdot y$, called \emph{multiplication}. A \emph{morphism of magmas} $f: (X, \, \cdot) \to (X', \, \cdot')$ is a function $f: X \to X'$ such that $f(x \cdot y) = f(x) \cdot' f(y)$, for all $x$, $y\in X$. We denote by ${\rm Mag}$ the category of all magmas. The forgetful functor to the category of sets
    \begin{equation*}
      {\rm Mag} \xrightarrow[]{\mathrm{forget}} {\rm Set}, \qquad (X, \, \cdot) \mapsto X
    \end{equation*}
admits a left adjoint $\mathcal{FM} (-) : {\rm Set} \to {\rm Mag}$, where for every set $X$, $\mathcal{FM} (X)$ denotes the free magma generated by $X$, as constructed in \cite[Theorem 1.2, Chapter IV]{serre} or \cite[Proposition 1, p. 81]{bourbaki}. In colloquial terms, $\mathcal{FM} (X)$ consists of all \emph{non-associative words} on $X$, containing all parentheses, with the multiplication given by the concatenation of words.\\
A \emph{congruence} on a magma $(X, \, \cdot)$ is an equivalence relation $\approx$ on $X$ compatible with the multiplication \cite[Definition 11, p. 11]{bourbaki}; that is, whenever $x \approx x'$ and $y \approx y'$ one has $x \cdot y \approx x' \cdot y'$. If $\approx$ is a congruence on $(X, \, \cdot)$, then the \emph{quotient magma} is the set $X/\approx$ equipped with the multiplication
$
 \overline{x} \cdot \overline{y} := \overline{x\cdot y}
$
for all $\overline{x}$, $\overline{y} \in X/\approx$. A magma $(X, \, \cdot)$ is called:
\begin{enumerate}
\item a \emph{left} (resp. \emph{right}) \emph{quasi-group} \cite{PPZ} if for any $a \in X$ the left (resp. right) translation map $l_{a}^{\cdot} : X \to X$, $l_{a}^{\cdot} (x) := a \cdot x$ (resp. $r_{a}^{\cdot} : X \to X$, $r_{a}^{\cdot} (x) := x \cdot a$) is bijective;
\item a \emph{quasi-group} if it is simultaneously a left and right quasi-group;
\item \emph{left} (resp. \emph{right}) \emph{cancellative} if for any $a \in X$ the left (resp. right) translation map $l_{a}^{\cdot} : X \to X$ (resp. $r_{a}^{\cdot} : X \to X$) is injective;
\item a \emph{right P{\l}onka magma} \cite[Definition 3.3]{ACMv3} if the multiplication satisfies the following two compatibility conditions for any $x$, $y$, $z\in X$
\begin{equation}\label{rkm1}
(x \cdot y) \cdot z = (x \cdot z) \cdot y,    \qquad x \cdot (y \cdot z) = x \cdot y.
\end{equation}
We denote by ${\rm Mag}^{\rm rP}$ the full subcategory of ${\rm Mag}$ of all right P{\l}onka magmas. The adjective \emph{right} used above comes from the fact that the two axioms of (\ref{rkm1}) can be rewritten equivalently using the right translation maps (see \cite[Remarks 3.4 (1)]{ACMv3} for more details);
\item a \emph{left P{\l}onka magma} \cite[Definition 3.6]{ACMv3} if $(X, \, \cdot^{op})$ is a right P{\l}onka magma or, in other words, if the multiplication satisfies the following compatibility conditions for any $x$, $y$, $z\in X$
\begin{equation}\label{remag}
x \cdot (y \cdot z) = y \cdot (x \cdot z),  \qquad (x \cdot y) \cdot z = y \cdot z.
\end{equation}
The full subcategory of ${\rm Mag}$ of all left P{\l}onka magmas will be denoted by ${}^{\rm lP}{\rm Mag}$.
\end{enumerate}
A \emph{bi-magma} is a triple $(X, \, \cdot, \, \ast)$ consisting of a
set $X$ with two binary operations $\cdot : X \times X \to X$ and $\ast : X \times X \to X$. As explained in the companion paper \cite[Introduction]{ACMv3}, there is a one-to-one correspondence between functions $R: X\times X \to X \times X$ and bi-magma structures on the underlying set $X$. We therefore work in the equivalent, and more convenient, setting of bi-magmas rather than with maps $R: X\times X \to X \times X$. Throughout, we implicitly assume that the map $R = R_{(\cdot,\, \ast)}$ associated with the bi-magma $(X, \, \cdot, \, \ast)$ is given for any $x$, $y \in X$ by:
\begin{equation}\label{notatieR}
R(x, \, y) = R_{(\cdot, \, \ast)} \, (x, \, y) = (x\cdot y, \, x \ast y)
\end{equation}
where we denote $x \cdot y = (\pi_1 \circ R) (x,y)$ and $x \ast y = (\pi_2 \circ R) (x,y)$. \\
A \emph{morphism of bi-magmas} $f: (X, \, \cdot, \, \ast) \to (X', \, \cdot', \, \ast')$ is a function $f: X \to X'$ such that $f(x \cdot y) = f(x) \cdot' f(y)$ and
$f(x \ast y) = f(x) \ast' f(y)$, for all $x$, $y\in X$. We denote by ${\rm BiMag}$ the category of all bi-magmas. It is a variety
of universal algebras of signature consisting of two binary operations, with an empty family of identities. It follows, using Corollary~\ref{cor:Freyd}, that the forgetful functors
    \begin{equation*}
      {\rm BiMag} \xrightarrow[]{\mathrm{forget}} {\rm Mag}, \quad (X, \, \cdot, \, \ast) \mapsto (X, \, \cdot), \qquad
      {\rm BiMag} \xrightarrow[]{\mathrm{forget}} {\rm Set}, \qquad (X, \, \cdot, \, \ast) \mapsto X
 \end{equation*}
each admit a left adjoint.

\begin{remark}\label{res:multenon}
The classical notion of non-degeneracy for a solution $R$ of the braid equation introduced in \cite[Definition 1.1]{etingof} admits a natural reformulation in terms of its associated bi-magma: the two underlying magmas must be, respectively, a left quasi-group and a right quasi-group. We termed this \emph{left-right non-degeneracy} in \cite[Remark 1.3]{ACMv3}, as it is one of a broader family of non-degeneracy conditions that may be imposed on a given solution of the braid equation. Via the aforementioned bijective correspondence between solutions of the Yang-Baxter and braid equations, non-degeneracy for a solution $R$ of the Yang-Baxter equation is characterised by the two associated magmas being, respectively, a right quasi-group and a left quasi-group. We refer to this condition as \emph{right-left non-degeneracy}. In the same vein, left non-degeneracy of a Yang-Baxter solution as defined in \cite{rump}, comes down to requiring the magma $(X, \, \cdot)$ to be a right quasi-group.
\end{remark}
Throughout, we denote by $\mathcal{BR}^{b,n}$ the full subcategory of $\mathcal{BR}$ consisting of all bijective and left-right non-degenerate\footnote{As explained in Remark~\ref{res:multenon}, we call \emph{left-right non-degenerate} a solution of the braid equation that is non-degenerate in the sense classically used in the literature (e.g. \cite{etingof}).} solutions of the braid equation. Similarly, $\mathcal{YB}^{\rm ln, u}$ stands for the full subcategory of $\mathcal{YB}$ consisting of all unitary left non-degenerate solutions of the Yang-Baxter equation as considered in \cite{rump}. \\
Let $(X, \, \cdot, \, \ast)$ be a bi-magma and $R = R_{(\cdot, \, \ast)}: X\times X \to X\times X$ the associated map given by (\ref{notatieR}). Then we can easily prove that $R = R_{(\cdot, \, \ast)}$ is a solution of the Yang-Baxter equation (\ref{qybe2}) if and only if the following compatibilities hold for all $x$, $y$, $z\in X$ (\cite[Lemma 2.1]{rusii}):
\begin{align}
  (x\cdot y) \cdot z &= \bigl(x \cdot (y \ast z) \bigl) \cdot (y \cdot z)\label{YB1}\\
  \bigl( x \cdot (y \ast z) \bigl) \ast (y\cdot z ) &= (x \ast y) \cdot \bigl( (x\cdot y) \ast z  \bigl)\label{YB2} \\
  x \ast (y \ast z) &=  (x \ast y) \ast \bigl( (x\cdot y) \ast z\bigl)\label{YB3}
\end{align}
A bi-magma $(X, \, \cdot, \, \ast)$ satisfying the compatibilities (\ref{YB1})-(\ref{YB3}) will be called a \emph{Yang-Baxter bi-magma} (\cite[Definition 1.4]{ACMv3}), and we denote by ${}_{\rm YB}{\rm BiMag}$ the full subcategory of ${\rm BiMag}$ consisting of all Yang-Baxter bi-magmas. The category ${}_{\rm YB}{\rm BiMag}$ is a variety of algebras in the sense of universal algebras and the functor
\begin{equation}\label{isoybbimag}
  {}_{\rm YB}{\rm BiMag} \to \mathcal{YB}, \qquad (X, \, \cdot, \, \ast) \mapsto (X, \, R_{(\cdot, \, \ast)})
\end{equation}
is of course an isomorphism of categories. In particular, we obtain that the category $\mathcal{YB}$ of all YB-solutions, being isomorphic to a variety of algebras, is complete and cocomplete and is a locally finitely presentable category \cite[Definition 1.9, Corollary 3.7]{adamek_ros}.\\
The notion of a \emph{P{\l}onka bi-magma} was introduced in \cite[Definition 3.7]{ACMv3} as a means of constructing new YB-solutions. A P{\l}onka bi-magma is a bi-magma $(X, \, \cdot, \, \ast)$ such that:

(PM1) $(X, \, \cdot)$ is a right P{\l}onka magma;

(PM2) $(X, \, \ast)$ is a left P{\l}onka magma;

(PM3) The following compatibility conditions hold for any $x$, $y$, $z\in X$:
\begin{align}
  x \ast (y \cdot z) &= (x \ast y) \cdot z\label{KM3a}\\
  (x \cdot z) \ast y  &= x \ast y\label{KM3b} \\
  x \cdot (y\ast z) &= x \cdot z.\label{KM3c}
\end{align}
We denote by ${\rm PBiMag}$ the full subcategory of ${\rm BiMag}$ consisting of all P{\l}onka bi-magmas. The main application of P{\l}onka bi-magmas is given in \cite[Theorem 3.10]{ACMv3} where new YB-solutions are constructed: if $(X, \, \cdot, \, \ast)$ is a P{\l}onka bi-magma then the associated canonical map defined for any $x$, $y\in X$ by
    \begin{equation}\label{eq:solgenK}
      R = R_{(\cdot, \, \ast)} : X\times X \to X\times X, \qquad R (x, \, y) = (x\cdot y, \, x\ast y)
    \end{equation}
is a YB-solution.\\
Furthermore, using again Corollary~\ref{cor:Freyd}, each of the forgetful functors
    \begin{align}
      {\rm PBiMag}\ni (X, \, \cdot, \, \ast) &\xmapsto[]{\quad\mathrm{forget}\quad}X\in {\rm Set}\label{for1k}\\
      {\rm PBiMag}\ni (X, \, \cdot, \, \ast) &\xmapsto[]{\quad\mathrm{forget}\quad}(X, \, \cdot)\in {\rm Mag}^{\rm rP}\label{for2k}\\
      % {\rm PBiMag}\ni (X, \, \cdot, \, \ast) &\xmapsto[]{\quad\mathrm{forget}\quad}(X, \, \ast)\in {\rm Mag}\label{for3k}\\
    {\rm PBiMag}\ni (X, \, \cdot, \, \ast) &\xmapsto[]{\quad\mathrm{forget}\quad}(X, \, \cdot, \, \ast)\in {\rm BiMag}\label{for4k}
  \end{align}
admits a left adjoint. %However, the explicit description of these left adjoints is rather cumbersome and we do not include it here. BA DA! PRIMUL ADJUNCT E spectaculos.
We recall, for further use, the structure of P{\l}onka bi-magmas as proved in \cite[Theorem 3.24]{ACMv3}:

\begin{theorem}\label{th:strpbimag}
  Let $X$ be a set. A bi-magma structure $(X, \, \cdot, \, \ast)$ on $X$ is a P{\l}onka bi-magma if and only if there exist:

  \begin{itemize}

  \item A  partition $(X_i)_{i\in I}$ on $X$, i.e. $X = \coprod_{i\in I} \, X_i$;

  \item A pair of tuples $\bigl(f_i^j \bigl)_{(i,j) \in I\times I}$ and $\bigl(g_i^j \bigl)_{(i,j) \in I\times I}$
    of functions $f_i^j : X_i \to X_i$ and $g_i^j : X_i \to X_i$ such that for any $i$, $j$, $k\in I$:
    \begin{equation}\label{eq:comfunctddd}
      f_i^j \circ f_i^k = f_i^k \circ f_i^j, \qquad g_i^j \circ g_i^k = g_i^k \circ g_i^j, \qquad g_i^j \circ f_i^k = f_i^k \circ g_i^j.
    \end{equation}
  \end{itemize}

  The P{\l}onka bi-magma $(X, \, \cdot, \, \ast)$ associated to
  $\bigl((X_i)_{i\in I}, \,  \bigl(f_i^j \bigl)_{(i,j) \in I\times I}, \, \bigl(g_i^j \bigl)_{(i,j) \in I\times I} \bigl)$ in the above correspondence is given by:
  \begin{equation}\label{eq:multRPmdd}
    x_i \cdot x_j := f_i^j (x_i), \qquad   x_i \ast x_j := g_j^i (x_j)
  \end{equation}
  for all $x_i \in X_i$, $x_j \in X_j$ and $i$, $j\in I$.
\end{theorem}
The correspondence in Theorem~\ref{th:strpbimag} is surjective but \emph{not} injective as shown in \cite[Remark 3.26]{ACMv3}.

%%%%%%%%%%%%%%%%%%%%%%%%%%%%%%%%%%%%%%%%%%%%%%%%%%%%%%%%%%%%%%%%
%%%%%%%%%%%%%%%%%%%%%%%%%%%%%%%%%%%%%%%%%%%%%%%%%%%%%%%%%%%%%%%%

\section{The Yang-Baxter equation and the structure of BLS-solutions}\label{se:bls}

In this section, we address Problem~B from the Introduction in the particular setting of YB-solutions induced by P{\l}onka bi-magmas as in (\ref{eq:solgenK}). More precisely, we are concerned with the following question:\, \emph{Which subcategory of $\mathcal{YB}$ is characterized by the family of solutions defined by (\ref{eq:solgenK})?} \\
Equivalently, in categorical terms we ask:

\emph{What is the image of the embedding functor ${\rm PBiMag} \lhook\joinrel\xrightarrow{\quad} \mathcal{YB}$, $(X, \, \cdot, \, \ast) \mapsto (X, \, R_{(\cdot, \, \ast)})$?}\\
To answer this question, we first recall the set-theoretic versions of three linear equations that arise in two seemingly unrelated contexts. The first two equations below were introduced by Baaj and Skandalis \cite[Definition 2.1]{Baaj} in connection with the duality theory of {$C^*$}-algebras. The third equation is the set-theoretic analogue of \cite[Definition 2.1]{militaruLong} and arose in the study of the category of Long dimodules over bialgebras, introduced by Long \cite{Long} in his work on the Brauer group.

\begin{definition}\label{def:longdef}
  Let $X$ be a set and $R \colon X \times X \to X\times X$ a map. Then $R$ is called:
  \begin{enumerate}[(1)]

  \item \emph{commutative} (resp. \emph{cocommutative}) if
    \begin{equation}\label{eq:bkcom}
      R^{12}R^{13} = R^{13}R^{12}, \qquad ({\rm resp.} \,\,\,  R^{13}R^{23} = R^{23}R^{13})
    \end{equation}

  \item a solution of the \emph{set-theoretic Long equation} if
    \begin{equation}\label{long}
      R^{12}R^{23} = R^{23}R^{12}
    \end{equation}

  \item a solution of the \emph{Baaj-Long-Skandalis} (BLS, for short) equation if it is simultaneously a commutative and a cocommutative solution of the Long equation.
  \end{enumerate}
\end{definition}
We denote by $\mathcal{BLS}$ the category of solutions of the BLS equation, whose objects are pairs $(X, \, R)$, where $X$ is a set and $R\colon X \times X \to X\times X$ is a solution of the three equations (\ref{eq:bkcom}) and (\ref{long}) on $X$. Such a map $R$ will be referred to as a \emph{BLS-solution} on $X$.

\begin{remarks}\label{longbefore}
  \begin{enumerate}[(1)]

  \item The Long equation (\ref{long}) is obtained from the Yang-Baxter equation (\ref{qybe2}) by deleting the middle term on both sides. This equation complements the equations (\ref{eq:bkcom}) by describing additional possible commutation relations between two distinct maps $R^{ij}$.

  \item The linear version of equation (\ref{long}) was first introduced and studied in \cite[Definition 2.1]{militaruLong} under the name of the  \emph{$\mathcal{D}$-equation}, and later in \cite[Chapter 7]{CMZ} under the name of the \emph{Long equation}. It was also investigated by Bichon and Street \cite{bichon_street} in the setting of monoidal categories. The latter terminology is motivated by the main result of \cite[Theorem 3.6]{militaruLong} (see also \cite[Theorem 74, p. 306]{CMZ}) which provides an FRT-type theorem for this equation: every (linear) solution of the Long equation arises as a canonical map associated with a Long dimodule over a canonical bialgebra constructed from the solution itself.  The category of Long dimodules was introduced in \cite{Long} in connection with the study of the Brauer group. For further details and motivation, we refer the reader to \cite{bichon_street, militaruLong} and \cite[Chapter 7]{CMZ}.

  \item A solution of the Long equation need not be a solution of the Yang-Baxter equation (see Examples~\ref{exs:solutiiLong} (\ref{longvs yb}) below). There is, however, a notable exception: any commutative and cocommutative solution of the Long equation (i.e. any BLS-solution) is a YB-solution. Indeed, if $R$ is a BLS-solution we obtain:
    \begin{align*}
      R^{12} R^{13} R^{23} \, \stackrel{(\ref{eq:bkcom})} = \,R^{13} R^{12} R^{23} \, \stackrel{(\ref{long})} = \, R^{13} R^{23} R^{12} \, \stackrel{(\ref{eq:bkcom})} = \, R^{23} R^{13} R^{12}
    \end{align*}
    which shows that $R$ is a YB-solution. Thus, we obtain a canonical embedding:
    \begin{equation}\label{eq:longinYB}
      \mathcal{BLS} \lhook\joinrel\xrightarrow{\quad}  \mathcal{YB}, \qquad (X, \, R) \mapsto (X, \, R)
    \end{equation}
  \end{enumerate}
\end{remarks}

\begin{examples}\label{exs:solutiiLong}
  \begin{enumerate}[(1)]

  \item For every set $X$, the identity map ${\rm Id}_{X\times X}$ provides a  BLS-solution. On the other hand, whenever $|X|\ge 2$, the flip map $\tau_X$ fails to satisfy the Long equation. Furthermore, the class of bijective BLS-solutions is closed under taking inverses: if $R$ is a bijective BLS-solution, then so is $R^{-1}$.

\item\label{longvs yb} Let $(X, \ast)$ be a magma and consider the map $R = R_{\ast} : X \times  X \to X \times X$ given by $R (x, \, y) = (x, \, x \ast y)$. Then $R_{\ast}$ is a solution of
the Long equation if and only if for all $x$, $y$, $z\in X$
$$(x \ast y) \ast z = y \ast z$$
that is, the second part of (\ref{remag}) holds. On the other hand, as proved in \cite[Example 2, p. 7]{dr}, $R_{\ast}$ is a solution of the Yang-Baxter equation if and only if for all $x$, $y$, $z\in X$
$$x \ast (y \ast z) = (x \ast y) \ast (x \ast z)$$
In particular, this shows that $R_{\ast}$ is simultaneously a solution of the Long and the Yang-Baxter equations if and only if $(X, \ast)$ is a left P{\l}onka magma.

\item\label{longvs yb2} Let $(X, \cdot)$ be a magma and $R = {}_{\cdot} R : X \times  X \to X \times X$ given by $R (x, \, y) = (x\cdot y, \, y)$.
Then $R$ is a solution of the Long equation if and only if for all $x$, $y$, $z\in X$
$$
x \cdot (y \cdot z) = x \cdot y
$$
and, $R$ is simultaneously a solution of the Long and the Yang-Baxter equations if and only if $(X,\cdot)$ is a right P{\l}onka magma.
\end{enumerate}
\end{examples}
With these concepts in place, we are now in a position to answer the question posed at the beginning of this section by showing that the subcategory $\mathcal{BLS} \subseteq \mathcal{YB}$ of all BLS-solutions
coincides with the family of YB-solutions which are described by (\ref{eq:solgenK}).

\begin{theorem}\label{teorema2}
The functor
\begin{equation}
{\rm PBiMag} \to  \mathcal{BLS} \subseteq \mathcal{YB}, \qquad (X, \, \cdot, \, \ast) \to (X, \, R_{(\cdot, \, \ast)})
\end{equation}
is an isomorphism of categories between the category of P{\l}onka bi-magmas and the category of BLS-solutions.
\end{theorem}

\begin{proof}
Let $(X, \, \cdot, \, \ast)$ be a bi-magma and $R = R_{(\cdot, \, \ast)} : X \times X \to X \times X$ the associated map given by (\ref{eq:solgenK}).
We will prove that $R$ is a solution of the BLS equation if and only if $(X, \, \cdot, \, \ast)$ is a P{\l}onka bi-magma, thereby completing the proof. Indeed, for all $x$, $y$, $z \in X$ we have:
\begin{equation*}
R^{12} R^{23} (x, \, y, \, z) = \bigl( x\cdot (y \cdot z), \, x\ast (y\cdot z), \, y\ast z \bigl)
\end{equation*}
and
\begin{equation*}
R^{23} R^{12} (x, \, y, \, z) = \bigl( x\cdot y , \, (x\ast y) \cdot z, \, (x \ast y) \ast z \bigl)
\end{equation*}
Hence, $R$ is a solution of the Long equation if and only if the following compatibilities hold
\begin{align}
x\cdot (y \cdot z) &=  x\cdot y\label{KM3aa}\\
(x \ast y) \ast z  &= y \ast z\label{KM3bb} \\
 x\ast (y\cdot z)  &= (x\ast y) \cdot z\label{KM3cc}
\end{align}
for all $x$, $y$, $z \in X$. These are the second identities in (\ref{rkm1}) and (\ref{remag}), together with (\ref{KM3a}), which constitute part of the
definition of a P{\l}onka bi-magma. In the same fashion we have:
\begin{equation*}
R^{12} R^{13} (x, \, y, \, z) = \bigl( (x\cdot z) \cdot y, \, (x \cdot z) \ast y, \, x\ast z \bigl)
\end{equation*}
and
\begin{equation*}
R^{13} R^{12} (x, \, y, \, z) = \bigl( (x\cdot y) \cdot z , \, x\ast y, \, (x \cdot y) \ast z \bigl)
\end{equation*}
Hence, $R$ is commutative if and only if the following compatibilities hold
\begin{equation*}
(x\cdot y) \cdot z = (x\cdot z) \cdot y, \qquad (x \cdot z) \ast y = x \ast y
\end{equation*}
for all $x$, $y$, $z \in X$. These correspond exactly to the first identities of (\ref{rkm1}) and (\ref{KM3b}), which appear in the definition of a P{\l}onka bi-magma. Finally, we can analogously see that $R$ is cocommutative if and only if the following compatibilities hold
\begin{equation*}
x\ast (y \ast z) = y\ast (x \ast z) , \qquad x \cdot (y \ast z) = x \cdot z
\end{equation*}
for all $x$, $y$, $z \in X$. These are precisely the first identities of (\ref{remag}) and (\ref{KM3c}) appearing in the definition of a P{\l}onka bi-magma. Thus, we have proved that $R = R_{(\cdot, \, \ast)}$ is a solution of the BLS equation if and only if $(X, \, \cdot, \, \ast)$ is a P{\l}onka bi-magma. Since both categories are full subcategories of ${\rm BiMag}$, the bijection on objects is automatically an isomorphism of categories,
completing the proof.
\end{proof}

\begin{remark}\label{res:lrcancela}
Let $(X, \, \cdot, \, \ast)$ be a bi-magma. The proof of Theorem~\ref{teorema2} shows that $R = R_{(\cdot, \, \ast)} : X \times X \to X \times X$ is a solution of the Long equation if and only if the compatibility conditions (\ref{KM3aa}) - (\ref{KM3cc}) hold. Hence, if $|X|\ge2$, $(X,\cdot)$ is not left cancellative: otherwise (\ref{KM3aa}) would give
$y\cdot z=y$, i.e.\ every left translation would be constant. In particular, no such solution is
left-right non-degenerate, while right-left non-degenerate BLS-solutions do exist (e.g.\ $R_{(f,g)}$
with $f$, $g$ commuting bijections, see Example~\ref{ex:bmlyb}).
\end{remark}
Combining Theorem~\ref{teorema2} and Theorem~\ref{th:strpbimag}, we obtain the main result of this section: an explicit combinatorial description of all solutions of the BLS equation
in terms of set partitions and maps.

\begin{theorem}\label{th:ultima}
  Let $X$ be a set and $R: X\times X \to X\times X$ a map. The following are equivalent:

  \begin{enumerate}[(a)]

  \item $R: X\times X \to X\times X$ is a solution of the BLS equation.

  \item The associated canonical bi-magma $(X, \, \cdot, \, \ast)$ is a P{\l}onka bi-magma.

  \item There exists a triple $\bigl((X_i)_{i\in I}, \,  \bigl(f_i^j \bigl)_{(i,j) \in I\times I}, \, \bigl(g_i^j \bigl)_{(i,j) \in I\times I} \bigl)$
    consisting of a partition on $X = \coprod_{i\in I} \, X_i$ and two tuples $\bigl(f_i^j \bigl)_{(i,j) \in I\times I}$ and $\bigl(g_i^j \bigl)_{(i,j) \in I\times I}$ of
    functions $f_i^j : X_i \to X_i$ and $g_i^j : X_i \to X_i$ satisfying the commutation relations
        \begin{equation}\label{comfunctddd}
      f_i^j \circ f_i^k = f_i^k \circ f_i^j, \qquad g_i^j \circ g_i^k = g_i^k \circ g_i^j, \qquad g_i^j \circ f_i^k = f_i^k \circ g_i^j.
    \end{equation}
for all $i$, $j$, $k\in I$. Moreover, $R$ is given for all $x_i \in X_i$, $x_j \in X_j$ and $i$, $j\in I$ by
    \begin{equation}\label{solgenbls}
      R(x_i, \, x_j) = \bigl( f_i^j (x_i), \,\,  g_j^i (x_j)\bigl).
    \end{equation}

  \end{enumerate}
\end{theorem}

\begin{proof}
The proof follows from Theorem~\ref{teorema2} and the structure theorem of P{\l}onka bi-magmas as given by Theorem~\ref{th:strpbimag}.
\end{proof}
Using Theorem~\ref{th:ultima}, a straightforward computation shows:

\begin{corollary}\label{co:divsol}
Let $R: X\times X \to X\times X$ be a BLS-solution as given by (\ref{solgenbls}). Then:
\begin{enumerate}[(1)]
\item\label{uitaregen} $R$ is unitary if and only if $f_i^j$ and $g_i^j$ are bijective and inverse to each other, for all $i$, $j\in I$.
\item $R$ is involutive if and only if $f_i^j : X_i \to X_i$ and $g_i^j : X_i \to X_i$ are involutions, for all $i$, $j\in I$.
\item $R$ is of diagonal type if and only if $f_i^{i} = g_i^{i} = {\rm Id}_{X_i}$, for all $i\in I$.
\end{enumerate}
\end{corollary}
For finite sets, another application of Theorem~\ref{th:ultima} together with \cite[Theorem 3.30 (3)]{ACMv3} yields:

\begin{corollary}\label{co:unitinde}
  Let $X$ be a finite set with $|X| = t$. Then the number of isomorphism classes of all ideal-simple (\cite[Definition 2.1]{ACMv3}) BLS-solutions on $X$ is $\sigma (t)$, the sum of all divisors of $t$.

%They are explicitly described and classified by the set of all triples $(m,\ n,\ d)\in \bZ_{>0}^3$ with $mn = t$ and $d \leq m$, identifying $X = \bZ/m\times \bZ/n$.
\end{corollary}
We end this section by introducing a new class of YB-solutions whose corresponding category will be proved to sit between $\mathcal{BLS}$ and the entire category $\mathcal{YB}$.

\subsection{Lyubashenko type solutions}\label{subse:lyb}
We recall the basic example of a YB-solution, attributed to Lyubashenko in \cite[Example 1]{dr}:

\begin{example}\label{ex:bmlyb}
Given two self-maps $f$, $g: X\to X$ the map defined for any $x$, $y\in X$ by
  \begin{equation}\label{eq:lysol}
    R = R_{(f, \, g)} : X\times X \to X\times X, \qquad R (x, \, y) := \bigl(f(x), \, g(y)\bigl)
\end{equation}
is a YB-solution if and only if $f\circ g = g\circ f$. The map $R_{(f, \, g)}$ given by (\ref{eq:lysol}) is called a \emph{Lyubashenko solution} on $X$.  The bi-magma on $X$, associated to $R_{(f, \, g)}$, will be denoted by $X_f^g = (X, \, \cdot_f, \, \ast^g)$, where $x\cdot_f \, y := f(x)$ and $x\ast^g y \, := g(y)$, for all $x$, $y\in X$ and we call it the \emph{Lyubashenko bi-magma} of $f$ and $g$. It can be easily seen that $X_f^g$ is a P{\l}onka bi-magma if $f\circ g=g\circ f$.\\
Obviously, the map $\tau_{X} \circ R_{(f,\, g)}(x,\, y) = (g(y),\, f(x))$ is a solution of the \emph{braid} equation if and only if $f\circ g = g\circ f$. 
\end{example}
Building on Example~\ref{ex:bmlyb}, we introduce the following new class of YB-solutions that generalizes the solution given in (\ref{eq:lysol}).

\begin{definition}\label{def:lymag}
  A YB-solution $(X, \, R) \in \mathcal{YB}$ is called of \emph{Lyubashenko type} if there exists a partition\footnote{We allow the trivial partition $X := X$, so the maps $R_{(f, \, g)}$ as given by (\ref{eq:lysol}) are of Lyubashenko type.} $X = \coprod_{i\in I} \, X_i$ of $X$ and two tuples $\bigl(f_i\bigl)_{i \in I}$ and $\bigl(g_i\bigl)_{i \in I}$ of functions $f_i$, $g_i : X_i \to X_i$ such that for any $i\in I$ and $x$, $y\in X_i$:
\begin{equation}\label{beleaua}
g_i \circ f_i = f_i \circ g_i, \qquad  R(x, \, y) = (f_i(x), \, g_i(y)).
\end{equation}
We denote by $\mathcal{LYB}$ the full subcategory of $\mathcal{YB}$ consisting of all Lyubashenko type solutions of the Yang-Baxter equation.
\end{definition}

\begin{remarks}\label{res:lybvsyb}
  \begin{enumerate}[(1)]
\item\label{co-co} We believe that the category $\mathcal{LYB}$, introduced in Definition~\ref{def:lymag}, provides a useful framework for understanding the category $\mathcal{YB}$ of YB-solutions. A first argument in this sense is that every \emph{diagonal} YB-solution\footnote{In the YB-literature this class of solutions was studied extensively under the name of \emph{square-free} solutions.}, i.e. $R(x, \, x) = (x, \, x)$, for all $x\in X$, is of Lyubashenko type. Indeed, one can simply consider the discrete partition into singletons $X = \coprod_{x\in X} \, \{x\}$ together with $f_x = g_x = {\rm Id}_{\{x\}}$, for all $x\in X$. Note, though, that not all YB-solutions are of Lyubashenko type: see Example~\ref{ex:notlyb} below.

     \item Any solution of the BLS equation is of Lyubashenko type in the sense of Definition~\ref{def:lymag}, i.e. we have a canonical embedding of categories:
  $\mathcal{BLS} \subseteq \mathcal{LYB} \subseteq \mathcal{YB}$. Indeed, by Theorem~\ref{th:ultima} take $f_i:=f_i^i$, $g_i:=g_i^i$. For strictness of $\mathcal{BLS}\subsetneq\mathcal{LYB}$,
  the flip $\tau_X$ ($|X|\ge2$) is diagonal but not a BLS-solution.

  \item\label{povextind} Assume that we have a partition $X = \coprod_{i\in I} \, X_i$ on a given set $X$. We assume, in addition, that we have
    Lyubashenko solutions $R_{(f_i, \, g_i)}$ as given by (\ref{eq:lysol}) on any component $X_i \subseteq X$ of the partition. This naturally leads to the following question:

    \emph{Is there a canonical way to extend the family of YB-solutions $R_{(f_i, \, g_i)}$, defined on the partition components $X_{i}$, to a YB-solution on the whole set $X$?}

  The problem itself is complicated: we need a way of \emph{gluing} these solutions to the whole set $X$ to obtain a YB-solution on $X$.  At first glance,  there seems to be no obvious canonical way to do this; however, for the special case of BLS-solutions, a method to do this is provided by Theorem~\ref{th:ultima}. 
\end{enumerate}
\end{remarks}

\begin{example}\label{ex:notlyb}
A family of (bijective, right-left non-degenerate) YB-solutions of \cite[Theorem 2.5 (i)]{etingof2} is not in $\mathcal{LYB}$. Keeping in mind that we work with solutions of the \emph{Yang-Baxter equation} whereas those displayed in loc.cit. are solutions to the \emph{braid equation}, they are of the form
  \begin{equation}\label{eq:xyh12}
    X^2\ni (x,x')\xmapsto{\quad R\quad}(x-qh,\, q^{-1}x'+(1-q^{-1})x+h)\in X^2
  \end{equation}
where $X = \bZ/p$, a prime-order cyclic group, $p\ge3$, $q\in(\mathbb{Z}/p)^\times$ with $q\neq1$ and $h\neq0$. Every non-empty $R$-invariant $X'\subseteq X$ (i.e. $R(X'\times X')\subseteq X'\times X'$) will be invariant under translation by $-qh$, and hence must be all of $X$. This shows that for such a YB-solution to be of Lyubashenko type, it would have to be a Lyubashenko solution (\ref{eq:lysol}). Clearly, this is not the case as the second component of $R(x,x')$ has a non-zero term depending on $x$ and therefore is {\it not} a function of the second component $x'$ of $(x,x')$.
\end{example}
Example~\ref{ex:notlyb} establishes not only that $\mathcal{LYB}$ is strictly contained in $\mathcal{YB}$, but it also shows that solutions of that form are not connected to $\mathcal{LYB}$ objects by non-trivial morphisms:

\begin{proposition}\label{pr:nomor}
  Consider a YB-solution $(X, \, R)$ of the form (\ref{eq:xyh12}) for prime-order $X = \bZ/p$, $p\ge3$, $q\in(\mathbb{Z}/p)^\times$, $h\ne 0$ and $q\ne 1$. Then:
  \begin{enumerate}[(1)]
  \item\label{L1} $(X, \, R)$ admits no morphisms from a non-empty $\mathcal{LYB}$ object.

  \item\label{L2} The only morphisms from $(X, \, R)$ to $\mathcal{LYB}$ objects factor through the singleton Lyubashenko solution $\{*\}_{\id}^{\id}$ in
  the notation of Example~\ref{ex:bmlyb}.
  \end{enumerate}
\end{proposition}

\begin{proof}
(1) Because (by definition) $\mathcal{LYB}$ restrict to Lyubashenko solutions, we may as well assume that there is a morphism
$$Y\xrightarrow{\quad\theta\quad}X$$
 from one such, say $Y_f^g$ (in the notation of Example~\ref{ex:bmlyb}), to $(X,R)$. The constraint that $\theta$ intertwines the two solutions means, in particular, that
    \begin{equation*}
q^{-1}\theta(y')+(1-q^{-1})\theta(y)+h=\theta\bigl(g(y')\bigl),\quad {\rm for\,\,all}\,\, y,\, y'\in Y
    \end{equation*}
    Since $1-q^{-1}$ is invertible because we assumed $q\ne1$ this immediately implies that $\theta$ must be constant, with single value, say, $x_{0} \in X$. But then
    \begin{equation*}
      (x_{0},x_{0}) = R(x_{0},x_{0}) = (x_{0} - qh,\, q^{-1}x_{0}+(1-q^{-1})x_{0}+h),
    \end{equation*}
would, in particular, imply $qh = 0$, contradicting the assumptions $h\ne0$ and $q\in(\mathbb{Z}/p)^\times$. \\
(2) This time, assume we have a morphism
    \begin{equation*}
      X\xrightarrow{\quad\theta\quad}Y
    \end{equation*}
    to an $\mathcal{LYB}$ object admitting a partition
    \begin{equation*}
      Y=\coprod_i Y_i,\quad Y_i\xrightarrow{\quad f_i,\ g_i\quad} Y_i
    \end{equation*}
    as in Definition~\ref{def:lymag}. If, say, $\theta(x)\in Y_i$, then the solution-intertwining property yields 
     \begin{equation*}
      \theta(x-qh) = f_i(\theta(x))\in Y_i
      \end{equation*}
    Because $h \neq 0$ and $q\in(\mathbb{Z}/p)^\times$, the translation by $-qh$ generates $\mathbb{Z}/p$ and therefore the image $\theta(X)$ is entirely contained in (one of the) $Y_i$. We can thus assume without loss of generality that $Y=Y_f^g$ was a Lyubashenko solution to begin with. The solution-intertwining property implies, in particular, that
    \begin{equation*}
    \theta\bigl(q^{-1}x' +(1-q^{-1})x+h\bigr)=g\bigl(\theta(x')\bigr),\quad {\rm for \,\,all}\,\, x,\,x' \in X.
         \end{equation*}
    This implies that $\theta$ is constant, and its single value $y_{0} \in Y$ must be a fixed point of both $f$ and $g$. We thus have
    \begin{equation*}
      \left\{\{y_{0}\},\ f|_{\{y_{0}\}},\ g|_{\{y_{0}\}}\right\}\cong \{*\}_{\id}^{\id},
    \end{equation*}
    finishing the proof.
\end{proof}
As an immediate consequence, it follows that the subcategory $\mathcal{LYB}$ of $\mathcal{YB}$ is \emph{too small} to recover the ambient $\mathcal{YB}$ in any of the 
standard category-theoretic senses of the word:

\begin{corollary}\label{cor:nolimcolim}
The $\mathcal{YB}$ objects of Example~\ref{ex:notlyb} are neither limits nor colimits taken in $\mathcal{YB}$ of diagrams in $\mathcal{LYB}$.
\end{corollary}
%%%%%%%%%%%%%%%%%%%%%%%%%%%%%%%%%%%%%%%%%%%%%%%%%%%%%%%%%%%%%%%%
%%%%%%%%%%%%%%%%%%%%%%%%%%%%%%%%%%%%%%%%%%%%%%%%%%%%%%%%%%%%%%%%

\section{A categorical framework around the Yang-Baxter (resp. braid) equation}\label{se:freyd}

This section develops a categorical framework for the Yang-Baxter (resp. braid) equation. We begin by revisiting the intensively studied class of left-right non-degenerate solutions of the braid equation (simply called \emph{non-degenerate solutions} in the literature \cite{etingof}).  A fundamental classification result for this class of solutions was established in \cite[Theorem 2.10]{etingof} where an equivalence between the category of left-right non-degenerate, bijective, and involutive solutions of the braid equation (called \emph{non-degenerate symmetric
sets} in loc. cit.) and the category of bijective cocycle quadruples is established. \\
Going beyond the setting considered in \cite{etingof} by removing the involutivity assumption, another notion that has attracted considerable attention in the study of left-right non-degenerate and bijective solutions of the braid equation over the past decade is that of a \emph{skew left brace}. Recall \cite[Definition 1.1]{smo_ven} that a \emph{skew left brace} is a bi-magma $(X, \, \cdot, \, \ast)$ such that $(X, \, \cdot)$ and $(X, \, \ast )$ are both groups and the following relation holds for all $x$, $y$, $z \in X$
\begin{equation}\label{compbrac}
x \ast  (y\cdot z) = (x \ast  y) \cdot x^{-1} \cdot (x \ast  z)
\end{equation}
where $x^{-1}$ denotes the inverse of $x$ in the group $(X, \, \cdot)$; the inverse of $x$ in the group $(X, \, \ast)$ will be denoted by $\overline{x}$.  We denote by ${\rm SLB}$ the full subcategory of ${\rm BiMag}$ consisting of all skew left braces.
If $(X, \, \cdot, \, \ast)$ is a skew left brace then the map defined for any $x$, $y\in X$ by:
\begin{equation}\label{solbrac}
\mathcal{R} = \mathcal{R}^{(\cdot, \, \ast)} : X \times X \to X \times X, \qquad
\mathcal{R} (x, \, y) :=  \bigl( x^{-1} \cdot (x\ast y), \, \overline{x^{-1} \cdot (x\ast y)} \ast  x \ast y \bigl)
\end{equation}
is a (left-right) non-degenerate and bijective solution of the braid equation \cite[Theorem 3.1]{vendramin1} whose inverse is precisely $\mathcal{R}^{(\cdot^{\rm op}, \, \ast)}$, where $(X, \, \cdot^{\rm op}, \, \ast)$ denotes the opposite skew left brace (\cite[Theorem 4.1]{koch}).  Remark that the map $\mathcal{R}^{(\cdot, \, \ast)}$ as defined by (\ref{solbrac}) is different from the canonical map associated to the bi-magma $(X, \, \cdot, \, \ast)$ in (\ref{notatieR}). Furthermore, we have a functor
\begin{equation}\label{sol_skew_br}
  \mathcal{S}: {\rm SLB} \to \mathcal{BR}^{b,n}, \qquad  \mathcal{S} \bigl ( (X, \, \cdot, \, \ast) \bigl) := (X, \, \mathcal{R}^{(\cdot, \, \ast)}), \quad  \mathcal{S}(\sigma) := \sigma
\end{equation}
That $ \mathcal{S}$ is indeed a functor follows readily by observing that any morphism of skew left braces preserves both group operations and their inverses, and therefore it obviously intertwines the induced solutions $\mathcal{R}^{(\cdot, \, \ast)}$ and $\mathcal{R}^{(\cdot ', \, \ast ')}$.

\begin{remark}\label{re:post-groups}
The concept of a skew left brace is one of several ways of recapturing a notion that already exists in the literature, namely that of a classical \emph{bijective $1$-cocycle}.
More precisely, the category ${\rm SLB}$ of skew left braces is equivalent to each of the following categories:

(1) \emph{bijective $1$-cocycles} with coefficients in a group, considered in \cite{etingof};

(2) \emph{braiding operators} on a group, as defined in \cite{LuYZ};

(3) \emph{braided groups} in the sense of \cite{tak_2001};

(4) \emph{post-groups}, introduced more recently in \cite{bai_guo}.  \\
For full details we refer to \cite[Theorem 2]{LuYZ}, \cite[Proposition 1.11]{vendramin1} and \cite[Theorems 3.18 and 3.25]{bai_guo}. Further properties of the category ${\rm SLB}$  have been studied in \cite{DBAFMP}.
\end{remark}
Our first result shows that there exist many bijective and (left-right) non-degenerate solutions of the braid equation that are not isomorphic to any
solution of the form $\mathcal{R}^{(\cdot,\ast)}$ associated with a skew left brace.

\begin{proposition}\label{prop_im}
The essential image\footnote{That is, the full subcategory of $\mathcal{BR}^{b,n}$ consisting
of all objects isomorphic to some $\mathcal{S} \bigl ( (X, \, \cdot, \, \ast) \bigl)$.} of the functor $\mathcal{S}: {\rm SLB} \to \mathcal{BR}^{b,n}$
is a proper subcategory of $\mathcal{BR}^{b,n}$.
\end{proposition}
\begin{proof}
Let $(X, \, \cdot, \, \ast)$ be a skew left brace and consider $\mathcal{R}^{(\cdot, \, \ast)}$ to be the induced solution of the braid equation defined by~(\ref{solbrac}). Recall that the two group structures underlying the skew left brace $(X, \, \cdot, \, \ast)$ share the same identity element \cite[Lemma 1.7 (1)]{vendramin1} which we will denote simply by $1$. Now a straightforward computation shows that $\mathcal{R}^{(\cdot, \, \ast)}(1,\, 1) = (1,\, 1)$. Furthermore, if $R$ is a solution of the braid equation isomorphic in $\mathcal{BR}$ to a solution of the form $\mathcal{R}^{(\cdot, \, \ast)}$ for a given skew left brace $(X, \, \cdot, \, \ast)$, then there exists some $x_{0} \in X$ such that $R(x_{0},\, x_{0}) = (x_{0},\, x_{0})$. Indeed, if $\sigma \colon \mathcal{R}^{(\cdot, \, \ast)} \to R$ is the said isomorphism in $\mathcal{BR}$ we have $$(\sigma \times \sigma) \circ \mathcal{R}^{(\cdot, \, \ast)}(1,\, 1) = R\bigl(\sigma(1),\sigma(1)\bigr)$$ which comes down to $(x_{0},\, x_{0}) = R(x_{0},\, x_{0})$ if we denote $x_{0} = \sigma(1)$. In order to conclude our claim, it suffices to construct a bijective and (left-right) non-degenerate solution $(X,\,R)$ of the braid equation such that $R(x,\, x) \neq (x,\, x)$ for all $x \in X$.   \\
To this end, let $f \colon X \to X$ be a bijective self-map on the set $X$ which admits no fixed points. Then the map $S \colon X \times X \to X \times X$ defined by $S(x,\,y) = \bigl(f(y),\, f(x)\bigl)$ is a bijective, left-right non-degenerate solution of the braid equation (see e.g. Example~\ref{ex:bmlyb}). Furthermore, as $f$ admits no fixed points we have $S(x,\, x) \neq (x,\, x)$ for all $x \in X$, as desired. 
 \end{proof}
As mentioned in the introduction, the Yang--Baxter literature already contains several categorical results, while others admit a categorical interpretation that has remained unexplored. For example, the bijective correspondence between cycle sets and left non-degenerate unitary YB-solutions established in \cite[Proposition 1]{rump} is, in fact, an isomorphism of categories, as we will show below. \\
We recall that a magma $(X, \, \cdot)$ is called a \emph{cycle set} \cite[Section 4]{rump} if it is a left quasi-group and the following condition holds for any $x$, $y$, $z\in X$
\begin{equation}\label{cycle}
(x \cdot y) \cdot (x \cdot z) = (y \cdot x) \cdot (y \cdot z).
\end{equation}
The full subcategory of ${\rm Mag}$ consisting of all cycle sets will be denoted by $\mathcal{CY}$. 
\begin{proposition}\label{cyc_iso}
The category $\mathcal{CY}$ of cycle sets is isomorphic to the full subcategory $\mathcal{YB}^{\rm ln, u}$ of $\mathcal{YB}$ consisting of all left non-degenerate unitary YB-solutions.
\end{proposition}
\begin{proof}
Given a left quasi-group $(X,\, \star)$ and $x$, $x' \in X$ we denote by $a(x,\, x') \in X$ the unique element satisfying $x \star a(x,\, x') = x'$.
Similarly, for a right quasi-group $(X,\, \cdot)$ and $x$, $x' \in X$ we denote by $b(x,\, x') \in X$ the unique element satisfying $b(x,\, x') \cdot x = x'$. We can now define the following functors 
\begin{align*}
& F \colon \mathcal{CY} \to \mathcal{YB}^{\rm ln, u},\quad F(X,\, \star) = (X,\, R_{\star}),\quad R_{\star}(x,\, y) := \bigl(a(y,\, x),\, a(y,\, x) \star y \bigl), \quad F(f) = f\\
& G \colon \mathcal{YB}^{\rm ln, u} \to \mathcal{CY},\quad G(X,\, R) = (X,\, \star), \quad x \star y := b(x,\, y), \quad G(g) = g
\end{align*}
Note that $b$ is well defined since left non-degeneracy of $R = R_{(\cdot,\, \ast)}$ makes $(X,\, \cdot)$ a right quasi-group. Moreover, by \cite[Proposition 1]{rump}, $F$, $G$ are well-defined on objects and mutually inverse. We are left to prove that both functors are well-defined on morphisms.\\
To this end, let $f \colon (X,\, \star) \to (Y, \bullet)$ be a morphism in $\mathcal{CY}$. We show that $f$ is also a morphism in $\mathcal{YB}^{\rm ln, u}$ between the induced YB-solutions
\begin{eqnarray*}
R_{\star}(x,\, x') = \bigl(a_{1}(x',\, x),\, a_{1}(x',\, x) \star x'\bigl) \,\,\, {\rm and}\,\,\, R_{\bullet}(y,\, y') = \bigl(a_{2}(y',\, y),\, a_{2}(y',\, y) \bullet y'\bigl)
\end{eqnarray*}
where $a_{1}(x',\, x) \in X$ and $a_{2}(y',\, y)\in Y$ denote the unique elements satisfying $x' \star a_{1}(x',\, x) =x$ and $x' \bullet a_{2}(y',\, y) = y$, respectively.
Applying $f$ to the defining relation $x' \star a_{1}(x',\, x) =x$ yields:
$$
f(x)\, = \, f\bigl(x' \star a_{1}(x',\, x)\bigl) \, =\, f(x') \bullet f\bigl(a_{1}(x',\, x)\bigl)
$$
where the last equality holds because $f$ is a morphism in $\mathcal{CY}$. This shows that $f\bigl(a_{1}(x',\, x)\bigl)$ solves the equation defining $a_{2}(f(x'),\, f(x))$ and therefore, by uniqueness, we obtain
\begin{equation}\label{C60}
f(a_{1}(x',\, x)) = a_{2}(f(x'),\, f(x))
\end{equation}
Furthermore, having in mind that $f$ is a morphism in $\mathcal{CY}$ we arrive at:
\begin{equation}\label{C61}
f(a_{1}(x',\, x) \star x') \, =\, f\bigl(a_{1}(x',\, x)\bigl) \bullet f(x') \, \stackrel{(\ref{C60})}{=} \, a_{2}(f(x'),\, f(x)) \bullet f(x')
\end{equation}
Now (\ref{C60}) and (\ref{C61}) show that $f$ is a morphism in $\mathcal{YB}^{\rm ln, u}$ from $F(X,\, \star) = R_{\star}$ to $F(Y,\, \bullet) = R_{\bullet}$:
\begin{align*}
(f \times f) \circ R_{\star}(x,\, x') &\,=\, \Bigl(f\bigl(a_{1}(x',\, x)\bigl),\, f\bigl(a_{1}(x',\, x) \star x'\bigl)\Bigl)\\
&\,=\, \Bigl(a_{2}\bigl(f(x'),\, f(x)\bigl),\, a_{2}\bigl((f(x'),\, f(x)\bigl) \bullet f(x')\bigl)\Bigl) \\
&\,=\, R_{\bullet} \circ(f \times f)(x,\, x')
\end{align*}
Assume now that $g \colon (X,\, R) \to (Y,\, S)$ is a morphism in $\mathcal{YB}^{\rm ln, u}$, where $R = R_{(\cdot,\, \ast)}$ and $S = S_{(\cdot',\, \ast')}$ are the YB-solutions induced by the bi-magmas $(X,\, \cdot, \, \ast)$ and $(Y,\, \cdot', \, \ast')$, respectively. Write $G(X,\, R) =  (X,\, \star)$ and $G(Y,\, S) = (Y,\, \star')$, where
\begin{equation}
x \star x' = b_{1}(x,\, x'), \quad y \star' y' = b_{2}(y,\, y'),\quad {\rm for\,\, all} \,\,x,\,x' \, \in \,X\,\, {\rm and}\,\, y,\, y'\, \in\, Y. \label{C80}
\end{equation}
Recall that $b_{1}(x,\, x') \in X$ and $b_{2}(y,\, y') \in Y$ are the unique elements satisfying $b_{1}(x,\, x') \cdot x = x'$ and $b_{2}(y,\, y') \cdot' y = y' $, respectively. Applying $g$, which is a morphism in $\mathcal{YB}^{\rm ln, u}$ and therefore a morphism between the associated bi-magmas, to the equation defining $b_{1}(x,\, x')$ yields
$$
g(x') \,=\, g\bigl(b_{1}(x,\, x') \cdot x\bigl)\, = \, g\bigl(b_{1}(x,\, x')\bigl) \cdot' g(x)
$$
This shows that $g\bigl(b_{1}(x,\, x')\bigl)$ solves the equation defining $b_{2}(g(x),\, g(x'))$ and therefore by uniqueness we get $g(b_{1}(x,\, x')) = b_{2}\bigl(g(x),\, g(x')\bigl)$.
Using (\ref{C80}) the last equality comes down to $g(x \star x') = g(x) \star' g(x')$ which proves that $g: (X,\, \star) \to (Y,\, \star')$ is a morphism in $\mathcal{CY}$, finishing the proof.
\end{proof}
Freyd's theorem (Corollary~\ref{cor:Freyd}) provides a framework for constructing \emph{universal objects} associated with the Yang-Baxter equation. We believe it deserves special attention for two reasons. First, one of the few constructions of this nature in the Yang-Baxter literature was established by Lu, Yan and Zhu \cite[Theorem 9]{LuYZ} for the category of Takeuchi's \emph{braided groups} \cite{tak_2001} (see also \cite[Definition 3.11]{bai_guo}). It was subsequently reformulated in the equivalent language of skew left braces by Smoktunowicz and Vendramin \cite[Theorem 4.5]{smo_ven} and later extended to the category of \emph{post-groups} by Bai, Guo, Sheng and Tang \cite[Theorem 3.20]{bai_guo}. Second, and more importantly, the existence of universal objects with a certain property in a given category (in our case, categories related to the Yang-Baxter equation), is a well-known tool for understanding the structure of the category itself. To begin with, we denote by:
\begin{equation*}
  \mathcal{A} :=  \{ {\rm Set}, \,  {\rm Mag}, \, {\rm BiMag}, \, {}_{\rm YB}{\rm BiMag}, \, \mathcal{YB}, \, \mathcal{BR}, \,
  {\rm Mag}^{\rm rP}, \, {\rm PBiMag}, \, \mathcal{BLS}, \, {\rm SLB}, \, {\rm PG} \}
\end{equation*}
the categories that are used throughout this paper, where ${\rm SLB}$ (resp. ${\rm PG}$) is the full subcategory of ${\rm BiMag}$ consisting of all
skew left braces \cite[Definition 1.1]{smo_ven} (resp. post-groups \cite[Definition 2.1]{bai_guo}). All the categories belonging to $\mathcal{A}$ are either varieties of
algebras or are equivalent to a variety of algebras. Thus, Corollary~\ref{cor:Freyd} applies to them as equivalences of categories obviously preserve adjoint functors.\\
However, there is a notable exception in the above picture: the category $\mathcal{LYB}$ of Lyubashenko type solutions of the Yang-Baxter equation (Definition~\ref{def:lymag}) is not a variety of algebras. A rigorous proof of this statement is beyond the scope of this paper and will not be included here. The underlying intuition behind the proof is that a partition of a set cannot be thought of as $n$-ary operations.\\
Freyd's theorem applied to any pair of categories of $\mathcal{A}$ reads:

\begin{corollary}\label{co:metafreyd}
  Let $\mathcal{C}$, $\mathcal{D} \in \mathcal{A}$ be two categories belonging to $\mathcal{A}$. Then any functor $F: \mathcal{C} \to \mathcal{D}$ which respects underlying sets has a left adjoint.  \qedhere
\end{corollary}
Among several special cases of Corollary~\ref{co:metafreyd} some of them deserve to be stated explicitly and we will do so below. Furthermore, to make them accessible to a wide audience we restate the results in the equivalent language of \emph{universal objects and arrows} \cite[Theorem 2, p. 83]{mcl}. The first application of Corollary~\ref{co:metafreyd} refers to the existence of a \emph{free YB-solution} on any set, i.e. the forgetful functor
\begin{equation*}
\mathcal{YB} \to  {\rm Set} , \qquad (X, \, R)  \mapsto X
\end{equation*}
has a left adjoint $\mathcal{F}^{\rm yb} : {\rm Set} \to \mathcal{YB}$ which is a faithful functor. Equivalently restated, we have:

\begin{corollary}\label{cor:freesolgen}
  For any set $S$, there exists a pair $\Bigl( \bigl(\mathcal{F}^{\rm yb} (S), \, \mathcal{R}_S \bigl), \, i_S \Bigl)$ consisting of a YB-solution $\bigl(\mathcal{F}^{\rm yb} (S), \, \mathcal{R}_S \bigl) \in \mathcal{YB}$ and an injective map $i_{S} : S \to \mathcal{F}^{\rm yb} (S)$ with the following universal property: \\
For any YB-solution $(X, \, R) \in \mathcal{YB}$ and any function $f: S \to X$, there exists a unique morphism $g : \bigl(\mathcal{F}^{\rm yb} (S), \, \mathcal{R}_S \bigl) \to (X, \, R)$ in the category $\mathcal{YB}$ such that the following diagram is commutative:
\begin{equation*}
  \begin{tikzpicture}[>=stealth,auto,baseline=(current  bounding  box.center)]
    \path[anchor=base]
    (0,0) node (l) {$S$}
    +(4,0) node (r) {$\mathcal{F}^{\rm yb} (S)$}
    +(2,-1.5) node (d) {$X$}
    ;
    \draw[->] (l) to[bend left=6] node[pos=.5,auto] {$i_{S}$} (r);
    \draw[->] (l) to[bend right=6] node[pos=.5,auto,swap] {$f $} (d);
    \draw[->] (r) to[bend left=6] node[pos=.5,auto] {$g $} (d);
  \end{tikzpicture}
  \quad \quad \quad {\rm i.e.}\,\,\, g \circ i_{S} = f
\end{equation*}
$\bigl(\mathcal{F}^{\rm yb} (S), \, \mathcal{R}_S \bigl)$ is called the \emph{free YB-solution generated by $S$}.
\end{corollary}

\begin{proof}
  The only thing we have to prove is that the unit of adjunction $i_S: S \to \mathcal{F}^{\rm yb} (S)$ is injective. Indeed, by taking the solution $(X, \, R) := (S, \, {\rm Id}_{S^2})$ and $f := {\rm Id}_{S}$ in the universal property above, we obtain that $i_S$ is injective. Thus the left adjoint $\mathcal{F}^{\rm yb} : {\rm Set} \to \mathcal{YB}$ is a faithful functor since the unit components are monomorphisms \cite[Lemma 3.6.6 (1)]{Ago}.
\end{proof}

\begin{remarks}\label{constrfreemag}
  \begin{enumerate}[(1)]
    
  \item Identifying $\mathcal{YB} \cong {}_{\rm YB}{\rm BiMag}$, the explicit construction of $\bigl(\mathcal{F}^{\rm yb} (S), \, \mathcal{R}_S \bigl)$
    for a given set $S$, although it is very technical and difficult to apply, follows a standard recipe from the universal algebra view-point. $\mathcal{F}^{\rm yb} (S)$ is the quotient of the free bi-magma $(\mathcal{F} (S), \, \cdot, \, \ast)$ generated by $S$, through the smallest congruence relation on $\mathcal{F} (S)$ which identifies all the non-associative words given by the equations (\ref{YB1})-(\ref{YB3}). The solution $\mathcal{R}_S$ on $\mathcal{F}^{\rm yb} (S)$ is the canonical one associated with a bi-magma as given by (\ref{notatieR}).
  \item The first application of Corollary~\ref{cor:freesolgen} is typical of a free object in a category. Any YB-solution $(X, \, R)$ is a quotient of the
    free YB-solution on $X$. Indeed, by taking $S := X$ and $f : = {\rm Id}_X$ in Corollary~\ref{cor:freesolgen}, we obtain that there exists a unique morphism $g: \bigl(\mathcal{F}^{\rm yb} (X), \, \mathcal{R}_X \bigl) \to (X, \, R)$ in the category $\mathcal{YB}$ such that $g \circ i_X = {\rm Id}_X$, i.e. $g$ is a surjective map. Moreover, the YB-solution $R$ is recovered as
    \begin{equation}\label{formsolgenfree}
      R = (g \times g) \circ \mathcal{R}_X \circ (i_X \times i_X)
    \end{equation}
    where $i_X : X  \lhook\joinrel\xrightarrow{\quad} \mathcal{F}^{\rm yb} (X)$ is the canonical embedding and $\mathcal{R}_X$ is the free YB-solution generated by the set $X$.

  \item Related to Corollary~\ref{cor:freesolgen} and as another application of Corollary~\ref{co:metafreyd}, it is interesting to point out that the canonical functors
    \begin{equation*}
      {\rm Set} \to \mathcal{YB}, \qquad X \mapsto (X, \, {\rm Id}_{X\times X}), \quad  \bigl( {\rm resp.} \, \,  X \mapsto (X, \, {\rm \tau}_X) \bigl)
    \end{equation*}
    also have left adjoints.

  \item\label{bijeuniv} Usually in the Yang-Baxter literature, only bijective solutions are considered \cite{etingof, LuYZ}. Applying
    Corollary~\ref{co:metafreyd} we obtain that there exists a free bijective YB-solution on any YB-solution $(X, \, R)$. We will elaborate on this statement.

    Let ${\rm Set}^{\bullet}$ be the category of sets with an endomorphism: i.e. objects in ${\rm Set}^{\bullet}$ are pairs $(X, \, f)$, where $X$ is a set and
    $f: X \to X$ a function. A morphism $\alpha : (X, \, f) \to (X', \, f')$ in ${\rm Set}^{\bullet}$ is a function $\alpha : X \to X'$ such that $\alpha \circ f = f' \circ \alpha$. Let ${\rm Set}^{\bullet}_{\rm bij}$ be the full subcategory of ${\rm Set}^{\bullet}$ consisting of all pairs $(X, \, f)$, where $f: X \to X$ is bijective.

    In the same fashion we denote $\mathcal{YB}_{\rm bij}$ the full subcategory of $\mathcal{YB}$ consisting of all pairs $(X, \, R)$, where $X$ is a set and
    $R$ a bijective YB-solution on $X$. These four categories are (isomorphic to) varieties of algebras, thus applying once again Corollary~\ref{co:metafreyd} we obtain that the
    canonical inclusions
    \begin{equation*}
      {\rm Set}^{\bullet}_{\rm bij} \lhook\joinrel\xrightarrow{\quad} {\rm Set}^{\bullet}, \qquad \mathcal{YB}_{\rm bij} \lhook\joinrel\xrightarrow{\quad} \mathcal{YB}
    \end{equation*}
    have left adjoints. In the former case, for instance, for an arbitrary self-map $X\xrightarrow{f}X$, the unit
    \begin{equation*}
      \begin{tikzpicture}[auto,baseline=(current  bounding  box.center)]
        \path[anchor=base]
        (0,0) node (l) {$X$}
        +(2,.5) node (u) {$X$}
        +(2,-.5) node (d) {$Y$}
        +(4,0) node (r) {$Y$}
        ;
        \draw[->] (l) to[bend left=6] node[pos=.5,auto] {$\scriptstyle f$} (u);
        \draw[->] (u) to[bend left=6] node[pos=.5,auto] {$\scriptstyle u$} (r);
        \draw[->] (l) to[bend right=6] node[pos=.5,auto,swap] {$\scriptstyle u$} (d);
        \draw[->] (d) to[bend right=6] node[pos=.5,auto,swap] {$\scriptstyle g$} (r);
      \end{tikzpicture}
    \end{equation*}
    of the adjunction can be described as follows:
    \begin{itemize}
    \item $Y$ is the colimit in ${\rm Set}$ of the diagram
      \begin{equation}\label{eq:xfxfx}
        X
        \xrightarrow{\quad f\quad}
        X\xrightarrow{\quad f\quad} \cdots;
      \end{equation}
    \item $X\xrightarrow{u}Y$ is the structure map of the colimit, identifying $X$ with the leftmost object in (\ref{eq:xfxfx});
    \item $Y\xrightarrow{g}Y$ is the map induced by $f$ via the universal property of the colimit.
    \end{itemize}
    Yet more explicitly, one can recover $Y$ as the union of the tower
    \begin{equation*}
        \overline{X}
        \xrightarrow{\quad \overline{f}\quad}
        \overline{X}\xrightarrow{\quad \overline{f}\quad} \cdots,
      \end{equation*}
      where $X\xrightarrow{} \overline{X}$ is the quotient modulo the equivalence relation identifying $x$ and $y$ whenever $f^{n}(x)=f^{n}(y)$ for sufficiently large $n$ and $\overline{X}\xrightarrow{\overline{f}}\overline{X}$ the map induced by $f$.
  \end{enumerate}
\end{remarks}

A subcategory $\mathcal{C} \subseteq \mathcal{D}$ is \emph{reflective} \cite[\S IV.3]{mcl} if the inclusion functor $\mathcal{C} \lhook\joinrel\xrightarrow{\quad} \mathcal{D}$ has a left adjoint. Extending the second statement of Remarks~\ref{constrfreemag} (\ref{bijeuniv}), and based on Corollary~\ref{co:metafreyd}, we have

\begin{corollary}\label{co:BLSadj}
  Any class $\mathcal{T}$ of YB-solutions definable by equations or division operations involving $R$ is reflective in $\mathcal{YB}$.
  In particular, the subcategory $\mathcal{BLS}\subset \mathcal{YB}$ of all BLS-solutions is reflective.  \qedhere

\end{corollary}

\begin{remark}\label{re:aproximarea}
  The adjoint functors of Corollary~\ref{co:BLSadj} can be constructed by the generators-and-relations procedure \cite[Theorem 8.14]{bergman} familiar from universal algebras. To summarize: for a given YB-solution
  $(X, \, R)$, one takes the quotient of the associated Yang-Baxter bi-magma $(X, \, \cdot, \, \ast)$ through the smallest (finest) congruence containing the relations on $X$ which identifies all non-associative words corresponding to the equations defining the class $\cT$. \\
The units of the adjoint pairs of functors from Corollary~\ref{co:BLSadj} give the \emph{best approximation} of an arbitrary YB-solution $(X, \, R)$ with a YB-solution of class $\cT$.  
\end{remark}
As noted at the beginning of this section, one of the very few results in the Yang-Baxter literature in the spirit of Corollary~\ref{co:metafreyd} and Corollary~\ref{co:BLSadj} is \cite[Theorem 9]{LuYZ} which was established for \emph{braiding operators} on groups. An equivalent formulation in the language of skew left braces appears in \cite[Theorem 4.5]{smo_ven} and the same result was later proved in the setting of post-groups in \cite[Theorem 3.20]{bai_guo}. As skew left braces have already appeared above, we adopt this framework to restate \cite[Theorem 4.5]{smo_ven} in terms of adjoint functors:

\begin{theorem}
The functor $\mathcal{S}: {\rm SLB} \to \mathcal{BR}^{b,n}$ which sends any skew left brace $(X, \, \cdot, \, \ast)$ to the corresponding bijective and (left-right) non-degenerate solution $(X, \, \mathcal{R}^{(\cdot, \, \ast)})$ of the braid equation, has a left adjoint.
\end{theorem}
The proof of this result, originally given in \cite[Theorem 9]{LuYZ}, is constructive and does not rely on Freyd's theorem. Consequently, the construction which gives the left adjoint functor is rather involved. Using Corollary~\ref{co:metafreyd} we strengthen the above result by showing that the additional assumptions of bijectivity and left-right non-degeneracy, which were used in the aforementioned proof, are in fact unnecessary. Indeed, since ${\rm SLB}$ and $\mathcal{BR}$ are (isomorphic to) varieties of algebras, the functor
\begin{equation}\label{brabradi}
  \mathcal{S}: {\rm SLB} \to \mathcal{BR}, \qquad \mathcal{S} \bigl ( (X, \, \cdot, \, \ast) \bigl) := (X, \, \mathcal{R}^{(\cdot, \, \ast)}), \quad \mathcal{S} (\sigma) := \sigma
\end{equation}
(which respects underlying sets) has a left adjoint $\mathcal{B}: \mathcal{BR} \to {\rm SLB}$. We restate this result in terms of universal objects and arrows as follows:

\begin{corollary}\label{cor:primaadjreform}
  Let $(X, \, R) \in \mathcal{BR}$ be an arbitrary solution of the braid equation. Then there exists a pair $\bigl(\mathcal{B} (X, \, R), \, \eta_{(X,\,R)}  \bigl)$ consisting of a skew left brace $\bigl(\mathcal{B}(X, \, R), \, \bullet , \odot  \bigl)$ and a morphism $\eta_{(X, \, R)}: (X, \, R) \to \bigl(\mathcal{B} (X, \, R), \, \mathcal{R}^{(\bullet , \odot  \bigl)}  \bigl)$ in the category $\mathcal{BR}$ with the following property: \\
For any skew left brace $(B, \, \cdot, \, \ast)$ and any morphism $\theta : (X, \, R) \to \bigl(B, \, \mathcal{R}^{(\cdot, \, \ast)} \bigl)$ in $\mathcal{BR}$,
  there exists a unique morphism of skew left braces $u: \bigl ( \mathcal{B} (X, \, R), \, \bullet , \odot  \bigl) \to (B, \, \cdot, \, \ast)$ such that the
  following diagram is commutative:
  \begin{equation*}
    \begin{tikzpicture}[>=stealth,auto,baseline=(current  bounding  box.center)]
      \path[anchor=base]
      (0,0) node (l) {$X$}
      +(4,0) node (r) {$\mathcal{B}(X, \, R)$}
      +(2,-1.5) node (d) {$B$}
      ;
      \draw[->] (l) to[bend left=6] node[pos=.5,auto] {$\eta_{(X,\,R)}$} (r);
      \draw[->] (l) to[bend right=6] node[pos=.5,auto,swap] {$\theta $} (d);
      \draw[->] (r) to[bend left=6] node[pos=.5,auto] {$u $} (d);
    \end{tikzpicture}
    \quad \quad \quad {\rm i.e.} \,\,\, u \circ \eta_{(X,\,R)} = \theta
  \end{equation*}
\end{corollary}

\begin{remark}\label{re:chivendcazspecial}
The pair of adjoint functors $\mathcal{B} \dashv \mathcal{S} $ from Corollary~\ref{cor:primaadjreform} is far from yielding an equivalence of categories: Proposition \ref{prop_im} shows that
the functor $\mathcal{S}: {\rm SLB} \to \mathcal{BR}$ is not essentially surjective, even after restricting its codomain to the subcategory $\mathcal{BR}^{b,n}$. Furthermore,
the functor $\mathcal{B}: \mathcal{BR} \to {\rm SLB}$ is not faithful. Indeed, \cite[p. 17]{LuYZ}, provides an example of a solution $(X, R)$ of the braid equation for which the unit
$\eta: (X, \, R) \to \bigl(\mathcal{B} (X, \, R), \, \mathcal{R}^{(\bullet , \odot )} \bigl)$ of the adjunction $\mathcal{B} \dashv \mathcal{S}$ is not injective. This shows that $\mathcal{B}$ cannot be faithful \cite[Lemma 3.6.6 (1)]{Ago}.
\end{remark}

%%%%%%%%%%%%%%%%%%%%%%%%%%%%%%%%%%%%%%%%%%%%%%%%%%%%%%%%%%%%%%%%
%%%%%%%%%%%%%%%%%%%%%%%%%%%%%%%%%%%%%%%%%%%%%%%%%%%%%%%%%%%%%%%%

\section{Concluding remarks and open questions}\label{se:conclusion}

As explained in Section~\ref{se:prel}, the two magmas $(X, \, \cdot)$ and $(X, \, \ast)$ underlying a bi-magma $(X, \, \cdot, \, \ast)$ can independently be required to be a left quasi-group, a right quasi-group, or a quasi-group. The existing literature focuses almost exclusively on the study of  \emph{left-right non-degenerate} solutions of the braid equation \cite{etingof}, assuming that $(X, \, \cdot)$ is a left quasi-group and $(X, \, \ast)$ is a right quasi-group. Given the categorical isomorphism (\ref{iso_br_yb}), the analogous condition imposed in the study of the Yang-Baxter equation is the \emph{right-left non-degeneracy}, see e.g. \cite{rump}. Notably, the one-sided notion of left non-degeneracy, requiring only $(X, \, \cdot)$ to be a right quasi-group, was considered in \cite{rump} for the Yang-Baxter equation. However, various forms of \emph{non-degeneracy} remain unexplored, leaving the following question widely open: \vspace{0.5cm}\\
{\bf Problem C}:  \emph{Study solutions of the braid (resp. Yang--Baxter) equation satisfying other possible combinations of quasi-group conditions: for instance, requiring both $(X, \, \cdot)$ and $(X, \, \ast)$ to be left or right quasi-groups, or requiring one or both operations to be quasi-groups.}\vspace{0.5cm}\\
For each of these alternative non-degeneracy conditions, the two-step program proposed in this paper can be pursued:
\begin{enumerate}[1.]
\item Identify an algebraic structure, analogous to cycle sets or skew left braces, whose associated variety maps functorially onto the corresponding subcategory of $\mathcal{BR}$ (resp. $\mathcal{YB}$).
\item Determine the properties of this functor; that is, whether it captures the \emph{entire} subcategory of solutions satisfying the given non-degeneracy condition, or only a \emph{proper} subcategory thereof.
\end{enumerate}
 A closely related instance of Problem~B that Section~\ref{se:freyd} leaves open concerns the functor
$$
\mathcal{S}: {\rm SLB} \to \mathcal{BR}^{b,n}
$$
defined in~(\ref{sol_skew_br}). Proposition~\ref{prop_im} shows that the essential image of $\mathcal{S}$ is a \emph{proper} subcategory of $\mathcal{BR}^{b,n}$, by providing an example of a solution in $\mathcal{BR}^{b,n}$ which is not isomorphic to a solution induced by a skew left brace. However, this gives no information about which bijective and (left-right) non-degenerate solutions arise from a skew left brace. It would be worthwhile to determine, in terms of the associated bi-magma $(X, \, \cdot, \, \ast)$, necessary and sufficient conditions for $R \in \mathcal{BR}^{b,n}$ to be isomorphic to some $\mathcal{R}^{(\cdot', \, \ast')}$.

\addcontentsline{toc}{section}{References}
%\bibliography{bib}{}
%\bibliographystyle{plain}

\Addresses

\end{document}